\documentclass[12pt]{article}
\usepackage[utf8]{inputenc}\usepackage[english]{babel}
\usepackage[T1]{fontenc}\usepackage{amsmath}
\usepackage{amssymb}\usepackage{amsthm}
\usepackage{amsfonts}\usepackage{mathrsfs}
\usepackage{geometry}\usepackage{enumerate}

\DeclareMathOperator{\convw}{\xrightarrow[]{w}}

\newtheorem{theorem}{Theorem}[section] 
\newtheorem{definition}{Definition}[section]

\newtheorem{proposition}{Proposition}[section]
\newtheorem{lemma}{Lemma}[section]
\newtheorem{corollary}{Corollary}[section]
\makeatletter\renewcommand{\subsection}{\@startsection{subsection}{1}
{0pt}{3.25ex plus 1ex minus.2ex}{-1em}{\normalfont\normalsize\bf}}\makeatother

\begin{document}

\title{On ortho and disjointly compact operators acting to Fr\'echet spaces}
\author{Svetlana Gorokhova$^{1}$\\ 
\small $1$ Southern Mathematical Institute VSC RAS, Vladikavkaz, Russia}
\date{\empty}
\maketitle

\begin{abstract}
{We study compactness along orthonormal (disjoint bounded) sequences for operators 
from Hilbert spaces (Banach lattices) to Fr\'echet spaces.}
\end{abstract}



\bigskip
\section{Introduction}

\hspace{4mm} 
Several kinds of compactness along orthonormal sequences were studied recently in \cite{G2024} for operators from a Hilbert space to a Banach space. 
It was followed by the investigation of compactness along bounded disjoint sequences in a Banach lattice by Emelyanov,  Erkur\c{s}un-\"Ozcan, and the author in \cite{EEG2025-2}.
In the present paper, some of results of \cite{G2024,EEG2025-2} are extended to operators whose codomains are Fr\'echet spaces.

This paper is organized as follows. Section 2 contains necessary preliminaries.
Section 3 is concerned with an othonormal criteria (Theorems \ref{a-theorem-G2024} and \ref{theorem-G2024}) of compactness of operators from a Hilbert space to a Fr\'echet space.
In section 4, we prove Theorem \ref{a-theorem-EEG2025-2} that characterizes operators from a Banach lattice to a Fr\'echet space that are 
compact along disjoint bounded sequences.

Throughout the paper: vector spaces are real; operators are linear; $\mathcal{H}$ stands for an infinite dimensional Hilbert space;
$X$ a Banach space; $B_X$ the closed unit ball of X; $Y=(Y,\tau)$ a locally convex space; ${\cal F}=({\cal F},\tau,p)$ a Fr\'echet space, whose 
locally convex topology $\tau$ is metrizable (in particular, $\tau$ is a Hausdorff topology) by a complete metric $p$; 
$E$ a Banach lattice, and $E_+$ its positive cone. By a \text{\rm d-bdd}-sequence in $E$ we
understand a disjoint bounded sequence. For further unexplained terminology and notation, we refer to \cite{ABo2006,AB2006,BD1984,Halmos1982}.

\bigskip
\section{Preliminaries}

\hspace{4mm} 
It is well known (see, for example, \cite{Halmos1982}) that a bounded linear operator between 
Hilbert spaces is compact iff it takes orthonormal sequences to norm-null ones. 
This was extended in \cite{G2024} as follows.

\begin{proposition}\label{oc-prop}
{\em (\cite[Theorem 2.3]{G2024}) An operator $T$ from $\mathcal{H}$ to $X$ is compact iff $T$ maps every orthonormal basis of $\mathcal{H}$ into a relatively compact set iff
$T$ maps every orthonormal sequence of $\mathcal{H}$ to norm-null one.}
\end{proposition}

\noindent
By considering an operator from \cite[p.292]{Halmos1982}, one can see that it is not possible to replace "maps every orthonormal basis"
by "maps some orthonormal basis" in Proposition \ref{oc-prop}. 
The key ingredient of the proof of Theorem 2.3 in \cite{G2024} is the following proposition.

\begin{proposition}\label{b-prop-G2024}
{\em (\cite[Lemma 1.2]{G2024}) An operator $T:\mathcal{H}\to X$ is bounded iff $T$ is bounded along orthonormal sequences.}
\end{proposition}

In order to extend Proposition \ref{oc-prop} to operators from Hilbert spaces to Fr\'echet spaces in Theorem \ref{a-theorem-G2024}, 
we shall use the following generalization of Proposition \ref{b-prop-G2024}, which is  interesting itself. 

\begin{theorem}\label{set bounded to TVS}
{\em Let ${\cal T}$ be a set of operators from $\mathcal{H}$ to $Y$. The following conditions are equivalent.
\begin{enumerate}[$i)$]
\item
${\cal T}$ is topologically bounded, that is the set ${\cal T}B_{\mathcal{H}}=\bigcup\limits_{T\in{\cal T}}T(B_\mathcal{H})$ is $\tau$-bounded.
\item
${\cal T}{\cal E}$ is topologically bounded for every orthonormal basis ${\cal E}$ of $\mathcal{H}$.
\item
${\cal T}$ is topologically bounded along orthonormal sequences, i.e., $\{Tx_n\}_{T\in{\cal T}, n\in\mathbb{N}}$ is $\tau$-bounded 
for every orthonormal sequence $(x_n)$ in $\mathcal{H}$.
\item
${\cal T}$ is topologically bounded along bounded orthonogonal sequences.
\end{enumerate}
If, additionally, $(Y,\tau)$ is metrizable, i.e., there exists a separating family $\{p_k\}_{k\in\mathbb{N}}$ of $\tau$-continuous seminorms on $Y$
(in this case $p(x,y)=\sum\limits_{k=1}^{\infty}\frac{2^{-k}p_k(x-y)}{1+p_k(x-y)}$ is a metric that induces the topology $\tau$), then the above conditions are equivalent to
\begin{enumerate}
\item[$v)$]
${\cal T}$ is equi-$p$-continuous, that is, for every $\varepsilon>0$ there exists $\delta>0$ such that $\sup\limits_{T\in{\cal T}}p(Tx,Ty)\le\varepsilon$
whenever $\|x-y\|\le\delta$.
\end{enumerate}}
\end{theorem}

\begin{proof}
The implications $i)\Longrightarrow ii)\Longrightarrow iii)$ and equivalence $iii)\Longleftrightarrow iv)$ are trivial.

\medskip
$iii)\Longrightarrow i)$:\
Let ${\cal T}$ be topologically bounded along orthonormal sequences, yet
${\cal T}$ is not topologically bounded. Then there exists a circled convex $\tau$-neighborhood $U$ of zero such that $\lambda{\cal T}B_{\mathcal{H}}\not\subseteq U$ 
for every $\lambda>0$. So, there exist a linearly independent normalized sequence $(x_n)$ in $\mathcal{H}$ and a sequence $(T_n)$ in ${\cal T}$
such that $T_nx_n\notin 2^nU$ for all $n$.
By the Gram--Schmidt orthonormalization, pick an orthonormal sequence $(u_n)$ in $\mathcal{H}$ such that $x_n\in\text{\rm span}\{u_i\}_{i=1}^n$ for all $n$.
By the assumption, there exists $\alpha>0$ satisfying $\alpha Tu_n\in U$ for all $T\in{\cal T}$ and $n\in\mathbb{N}$. Then 
$$
   2^nU\not\ni T_nx_n=\sum\limits_{k=1}^{n}(x_n,u_k)T_nu_k\in n\alpha^{-1}U 
   \ \ \ \ (\forall n \in\mathbb{N}),
$$
which is absurd. The obtained contradiction completes the proof of the implication.

\medskip
Now, assume additionally that $\tau$ is induced by the metric $p(x,y)=\sum\limits_{k=1}^{\infty}\frac{2^{-k}p_k(x-y)}{1+p_k(x-y)}$.

\medskip
$i)\Longrightarrow v)$:\
Let $\varepsilon>0$ and let ${\cal T}B_{\mathcal{H}}$ be $\tau$-bounded. So, each semi-norm $p_k$ is bounded on ${\cal T}B_{\mathcal{H}}$, say
$\sup\limits_{T\in{\cal T};\|x\|\le 1}p_k(Tx)\le M_k$ for every $k\in\mathbb{N}$ (without lost of generality, $M_k\le M_{k+1}$). 
Take $n$ such that $\frac{1}{2^n}\le\frac{\varepsilon}{2}$. Then
$$
   \sup\limits_{T\in{\cal T};\|x-y\|\le\delta}p(Tx,Ty)=\sup\limits_{T\in{\cal T};\|x-y\|\le\delta}\sum\limits_{k=1}^{\infty}\frac{2^{-k}p_k(T(x-y))}{1+p_k(T(x-y))}\le
   \sum\limits_{k=1}^{\infty}\frac{2^{-k}\delta M_k}{1+\delta M_k}\le
$$
$$
   \frac{1}{2^n}+\sum\limits_{k=1}^{n}2^{-k}\delta M_k\le\frac{1}{2^n}+\frac{1}{2}\delta M_n\le\varepsilon
$$
for every $\delta$ with $0<\delta M_n\le\varepsilon$. Therefore, ${\cal T}$ is equi-$p$-continuous.

\medskip
$v)\Longrightarrow i)$:\
Let ${\cal T}$ be equi-$p$-continuous.
Assume in the contrary ${\cal T}B_{\mathcal{H}}$ is not $\tau$-bounded. Then, for some $k_0$, the semi-norm $p_{k_0}$ is not bounded on ${\cal T}B_{\mathcal{H}}$.
So, for each $\delta>0$, there exist $x_0\in\delta B_{\mathcal{H}}$  and $T\in{\cal T}$ such that $\frac{p_{k_0}(Tx_0)}{1+p_{k_0}(Tx_0)}\ge\frac{1}{2}$. 
Thus, $\sup\limits_{T\in{\cal T};\|x\|\le\delta}p(Tx,0)\ge{2^{-k_0-1}}$, a contradiction.
\end{proof}

\noindent
The following corollary of Theorem \ref{set bounded to TVS} generalizes \cite[Lemma1.2]{G2024}. 

\begin{corollary}\label{bounded to TVS}
{\em An operator $T$ from a Hilbert space to a locally convex space is topologically bounded iff $T$ carries orthonormal basises
onto topologically bounded sets iff $T$ is topologically bounded along bounded orthonogonal sequences iff $T$ is topologically bounded 
along othonormal sequences.}
\end{corollary}

We shall use also the next lemma that follows easily from \cite[Lemma 1.1]{G2024}. 

\begin{lemma}\label{wc-prop}
{\em Each weakly null sequence $(x_n)$ in $\mathcal{H}$
contains a subsequence $(x_{n_k})$ such that $|(x_{n_{k_1}},x_{n_{k_2}})|\to 0$ as $k_1\ne k_2$ and $k_1,k_2\to\infty$.}
\end{lemma}

\section{Compactness along orthonormal sequences}

\hspace{4mm}
We begin with an extension of \cite[Theorem 2.2 i)]{G2024} to operators from a Hilbert space to a Fr\'echet space. 
Note that the proof of Theorem~2.2~i) in \cite{G2024} contains a gap, which is fixed in the proof of Theorem \ref{a-theorem-G2024} below.
In this and in the next section, $({\cal F},\tau,p)$ is a Fr\'echet space such that $p(x,y)=\sum\limits_{k=1}^{\infty}\frac{2^{-k}p_k(x-y)}{1+p_k(x-y)}$ for $x,y\in {\cal F}$, 
where the topology $\tau$ is induced by the family $\{p_k\}_{k\in\mathbb{N}}$ of semi-norms on ${\cal F}$.

\begin{theorem}\label{a-theorem-G2024}
{\em Let ${\cal T}$ be a set of operators from $\mathcal{H}$ to $({\cal F},\tau,p)$. The following conditions are equivalent.
\begin{enumerate}[$i)$]
\item
${\cal T}B_\mathcal{H}$ is a relatively compact subset of ${\cal F}$.
\item
${\cal T}{\cal E}$ is relatively compact for every orthonormal basis ${\cal E}$ of $\mathcal{H}$.
\item
${\cal T}$ is compact along orthonormal sequences, that is the set $\{Tx_n: T\in{\cal T}, n\in\mathbb{N}\}$ is relatively compact 
for every orthonormal sequence $(x_n)$ in $\mathcal{H}$.
\item
${\cal T}$ is compact along bounded orthogonal sequences.
\end{enumerate}}
\end{theorem}

\begin{proof}
The implications $i)\Longrightarrow ii)\Longrightarrow iii)$ and the equivalence $iii)\Longleftrightarrow iv)$ are trivial.

\medskip
$iii)\Longrightarrow i)$:\
Assume in the contrary ${\cal T}B_\mathcal{H}$ is not relatively $\tau$-compact. 
Then, for some $\alpha>0$, there exist a sequence $(T_n)$ in ${\cal T}$ and a sequence $(y_n)$ in 
$B_\mathcal{H}$ with $p(T_ny_n,T_my_m)\ge 3\alpha$ for all $m\ne n$. 
Since $B_\mathcal{H}$ is relatively weakly compact, by passing to a subsequence, we may assume $y_n\convw y$.
Observe that the set ${\cal T}y$ is relatively compact by assumption $iii)$. 
So, by passing to a subsequence, we may assume $p(T_ny,u)\to 0$ for some $u\in{\cal F}$. 
By Theorem \ref{set bounded to TVS}, ${\cal T}$ is equi-$p$-continuous. So, for large enough $n$ and $m$, $n\ne m$, we have
$$
   p(T_n(y_n-y),T_m(y_m-y))=p(T_ny_n-T_ny,T_my_m-T_my)\ge
$$
$$
   p(T_ny_n-u,T_my_m-u)-\alpha=p(T_ny_n,T_my_m)-\alpha\ge 3\alpha-\alpha=2\alpha.
$$
Without lost of generality, we may assume $p(T_nz_n,T_mz_m)\ge 2\alpha$ for all $m\ne n$, where $z_n:=y_n-y\convw 0$.
Due to Lemma \ref{wc-prop}, by passing to a subsequence, we may assume $|(z_n,z_m)|\to 0$ as $n\ne m$ and $n,m\to\infty$.
By the weak lower semicontinuity of the norm (cf. \cite[Lemma\,6.22]{ABo2006}), 
$2\ge\|z_n\|=\|y_n-y\|\ge C$ for all $n\ge n_0$ and some $C>0$.
By passing to a further subsequence, we may assume $\|z_n\|\to\beta\in[C,2]$.
By the Gram--Schmidt orthogonalization, there exists an orthogonal sequence $(x_{n_i})$ 
in $\mathcal{H}$ with $\|x_{n_i}-z_{n_i}\|\to 0$ as $i\to\infty$. In particular, $\|x_{n_i}\|\to\beta$.
As ${\cal T}$ is equi-$p$-continuous, we may assume $\|x_{n_i}\|=\beta$ for all $i$ and $p(T_{n_i}x_{n_i},T_{n_j}x_{n_j})\ge\alpha$ for all $i\ne j$.

Letting $w_i=\frac{x_{n_i}}{\beta}$,  we have, for $i\ne j$,
$$
   p(T_{n_i}w_i,T_{n_j}w_j)=\sum\limits_{k=1}^{\infty}\frac{2^{-k}\beta^{-1}p_k(T_{n_i}x_{n_i}-T_{n_j}x_{n_j})}{1+\beta^{-1}p_k(T_{n_i}x_{n_i}-T_{n_j}x_{n_j})}\ge
$$
$$
   \sum\limits_{k=1}^{\infty}\frac{2^{-k}\beta^{-1}p_k(T_{n_i}x_{n_i}-T_{n_j}x_{n_j})}{(\beta+\beta^{-1})+(\beta+\beta^{-1})p_k(T_{n_i}x_{n_i}-T_{n_j}x_{n_j})}= 
$$
$$
   \frac{\beta^{-1}}{\beta+\beta^{-1}}\sum\limits_{k=1}^{\infty}\frac{2^{-k}p_k(T_{n_i}x_{n_i}-T_{n_j}x_{n_j})}{1+p_k(T_{n_i}x_{n_i}-T_{n_j}x_{n_j})}=  
   \frac{\beta^{-1}}{\beta+\beta^{-1}}p(T_{n_i}x_{n_i},T_{n_j}x_{n_j})\ge\frac{\alpha\beta^{-1}}{\beta+\beta^{-1}}>0.
$$
Thus, the set $\{Tw_i: T\in{\cal T}, i\in\mathbb{N}\}$ is not relatively compact, and hence ${\cal T}$ is not compact along orthonormal sequences. A contradiction.
\end{proof}

Now, we are ready to extend \cite[Theorem 2.3]{G2024} to operators from Hilbert spaces to Fr\'echet spaces.

\begin{theorem}\label{theorem-G2024}
{\em Let $T:\mathcal{H}\to ({\cal F},\tau,p)$. The following are equivalent.
\begin{enumerate}[$i)$]
\item
$T$ is compact in the sense that $T(B_\mathcal{H})$ is a relatively compact subset of ${\cal F}$.
\item
$T$ maps every orthonormal basis of $\mathcal{H}$ into a relatively compact subset of ${\cal F}$.
\item
$T$ maps every orthonormal sequence of $\mathcal{H}$ to a $p$-null sequence.
\item
$T$ maps every bounded orthogonal sequence of $\mathcal{H}$ to a $p$-null sequence.
\end{enumerate}}
\end{theorem}

\begin{proof}
The equivalence $i)\Longleftrightarrow ii)$ is a special case of Theorem \ref{a-theorem-G2024}, and the equivalence $iii)\Longleftrightarrow iv)$ is trivial.

\medskip
$ii)\Longrightarrow iii)$:\
By Theorem \ref{set bounded to TVS}, $T$ is $p$-continuous (and hence $\tau$-continuous).
Let $(u_n)$ be an orthonormal sequence of $\mathcal{H}$.
On the way to a contradiction, suppose $p(Tu_n,0)\not\to 0$.
By passing to a subsequence, we may suppose $p(Tu_n,0)\ge M$ for all $n$ and some $M>0$.
Extending $\{u_n\}_{n\in\mathbb{N}}$ to an orthonormal basis of $\mathcal{H}$, and involving $ii)$,
we see $\{Tu_n\}_{n\in\mathbb{N}}$ is relatively compact in ${\cal F}$, and hence $(Tu_n)$ has a $p$-convergent subsequence,
say $\lim\limits_{j\to\infty}p(Tu_{n_j},y)=0$ for some $y\in {\cal F}$. 
As $(u_{n_j})$ is orthonormal, $u_{n_j}\convw 0$. Since $T$ is $\tau$-continuous, $Tu_{n_j}\xrightarrow[]{\sigma({\cal F},{\cal F}')}0$. 
It follows from $Tu_{n_j}\xrightarrow[]{\tau} y$ that $Tu_{n_j}\xrightarrow[]{\sigma({\cal F},{\cal F}')}y$.
Therefore, $y=0$, and hence $\lim\limits_{j\to\infty}p(Tu_{n_j},0)=0$, contradicting $p(Tu_n,0)\ge M>0$ for all $n$.

\medskip
The implication $iii)\Longrightarrow ii)$ is trivial.
\end{proof}

Note that Proposition \ref{oc-prop} follows directly from Theorem \ref{theorem-G2024}.
It is worth noting that the statements of Theorems \ref{a-theorem-G2024} and \ref{theorem-G2024} are also valid for operators
from a Hilbert space to a locally convex space whose topology is metrizable. 

\bigskip
\section{Compactness along disjoint bounded sequences}

\hspace{4mm} 
In compare with Lemma \ref{bounded to TVS}, the case of operators acting from Banach lattices to locally convex spaces is more complicated. 
We do not know reasonable (certainly, the domain must be infinite-dimensional) conditions for automatic topological boundedness
of an operator $T:E\to X$ in terms of $\tau$-boundedness of $T$ along \text{\rm d-bdd}-sequences in $E$.
Recently, an example of a discontinuous linear functional on $\ell^p$ which is bounded along \text{\rm d-bdd}-sequences
was constructed by Storozhuk \cite{S2025}. Futhermore, the lattice-version of Theorem \ref{a-theorem-G2024} if failed for compactness along \text{\rm d-bdd}-sequences, 
even for bounded operators from a Banach lattice to a Banach space \cite[Example 1]{EEG2025-2}.

\begin{definition}\label{d-comp-def}
{\em (cf., \cite[Definition 4]{EEG2025-2})
A topologically bounded operator from a Banach lattice to a locally convex space is {\em\text{\rm d}-compact} 
whenever it carries \text{\rm d-bdd}-sequences to relatively compact sets.}
\end{definition}

\noindent
Examples like \cite[Example 1]{EEG2025-2} motivate the question: for which $E$ and $X$ there is a \text{\rm d}-compact operator 
$T:E\to X$ that is not compact?  The compactness along \text{\rm d-bdd}-sequences is equivalent to several kind of compactness along disjoint sets. 

\begin{proposition}\label{d-K-elem-1}
{\em (\cite[Proposition 2]{EEG2025-2}) Let $T:E\to X$ be bounded. The following are equivalent.
\begin{enumerate}[$i)$]
\item
$T$ is \text{\rm d}-compact.
\item
$T$ compact along \text{\rm d-bdd}-sequences of $E_+$.
\item
$T(D)$ is relatively compact for every disjoint normalized subset $D$ of $B_E$.
\item
$T(D)$ is relatively compact for every disjoint normalized subset $D$ of $E_+$.
\end{enumerate}}
\end{proposition}

\noindent
We are going to extend Proposition \ref{d-K-elem-1} for operators from Banach lattices to Fr\'echet spaces as follows.

\begin{theorem}\label{a-theorem-EEG2025-2}
{\em Let ${\cal T}$ be a set of topologically bounded operators from $E$ to $({\cal F},\tau,p)$. The following conditions are equivalent.
\begin{enumerate}[$i)$]
\item
${\cal T}D$ is relatively compact in ${\cal F}$ for every disjoint bounded subset $D$ of $E$.
\item
${\cal T}D$ is relatively compact in ${\cal F}$ for every disjoint normalized $D\subseteq E_+$.
\item
For every \text{\rm d-bdd}-sequence $(x_n)$ in $E$, the set $\{Tx_n: T\in{\cal T}, n\in\mathbb{N}\}$ is relatively compact in ${\cal F}$.
\item
For every \text{\rm d-bdd}-sequence $(x_n)$ in $E_+$, the set $\{Tx_n: T\in{\cal T}, n\in\mathbb{N}\}$ is relatively compact in ${\cal F}$.
\end{enumerate}}
\end{theorem}

\begin{proof}
Without lost of generality, assume $\dim(E)=\infty$.
Implications $i)\Longrightarrow ii)$ and $i)\Longrightarrow iii)\Longrightarrow iv)$ are trivial. So, it suffices to prove $ii)\Longrightarrow i)$ and $iv)\Longrightarrow ii)$.

\medskip
$ii)\Longrightarrow i)$:\
Let $D$ be a disjoint subset of $E$. Without lost of generality, $D\subseteq B_E$. Denote
$D_1=\{x_+: x\in D\ \& \ x_+\ne 0\}$ and $D_2=\{x_-: x\in D\ \& \ x_-\ne 0\}$.
If one them is empty, say $D_2=\emptyset$, then
$$
   {\cal T}D\subseteq{\cal T}D_1\cup\{0\}\subseteq\text{\rm cch}\bigl(\bigl\{T(\|x\|^{-1}x)\bigl\}_{T\in{\cal T}; x\in D_1}\bigl),
   \eqno(1)
$$
where $\text{\rm cch}(S)$ is the circled convex hull of $S$. The set $\bigl\{T(\|x\|^{-1}x)\bigl\}_{T\in{\cal T}; x\in D_1}$ is relatively compact in ${\cal F}$ due to $ii)$,
and hence ${\cal T}D$ is relatively compact by (1).

If $D_1$ and $D_2$ are both nonempty then $D\subseteq D_1- D_2$, and hence
$$
   {\cal T}D\subseteq\text{\rm cch}\bigl(\bigl\{T(\|x\|^{-1}x)\bigl\}_{T\in{\cal T}; x\in D_1}\bigl)+
   \text{\rm cch}\bigl(\bigl\{T(\|x\|^{-1}x)\bigl\}_{T\in{\cal T}; x\in D_2}\bigl).
   \eqno(2)
$$
It follows from (2) that ${\cal T}D$ is relatively compact in ${\cal F}$.

\medskip
$iv)\Longrightarrow ii)$:\
Let $(y_n)$ be a sequence in ${\cal T}D$, where $D$ is a disjoint normalized subset of $E_+$.
If $y_n=0$ for infinitely many $n$ then $(y_n)$ has a norm-null subsequence.
If $y_n=0$ only for finitely many $n$, by removing them, we may suppose all $y_n\ne 0$. For each $n$, pick $x_n\in D$ and $T_n\in{\cal T}$
with $y_n=T_nx_n$ and denote $D_0=\bigl\{x_n: n\in\mathbb{N}\bigl\}$. 
Then $D_0$ is a disjoint normalized countable (maybe even finite) subset of $E_+$.
By~$iv)$, ${\cal T}D_0$ is relatively compact. Since the sequence 
$(y_n)=(T_nx_n)$ lies in ${\cal T}D_0$, the sequence $(y_n)$ contains a $p$-convergent subsequence, and hence ${\cal T}D$ is relatively compact in ${\cal F}$.
\end{proof}

An important case of Definition \ref{d-comp-def} occurs (see \cite[Subsection 2.3]{EEG2025-2}) when a locally convex space in the codomain 
is a Banach space with its weak topology. The corresponding operators are called \text{\rm d}-weakly compact accordingly to \cite{EEG2025-2}. 
For a \text{\rm d}-weakly compact operator that is not weakly compact, we refer to \cite[Example 2]{EEG2025-2}.

\bibliographystyle{plain}

\addcontentsline{toc}{section}{KAYNAKLAR}
\begin{thebibliography}{99}
\normalsize	

\bibitem{ABo2006}
C. D. Aliprantis, K. Border, 
{\em Infinite dimensional analysis. A Hitchhiker’s Guide. 3-rd Edition.}
Springer (2006)

\bibitem{AB2006}
C. D. Aliprantis, O. Burkinshaw, 
{\em Positive operators.}
Springer (2006)

\bibitem{BD1984}
J. Bourgain, J. Diestel, 
{\em Limited operators and strict cosingularity.}
Math. Nachr., 119, 55--58 (1984)

\bibitem{G2024}
S. Gorokhova,  
{\em On compact (limited) operators between Hilbert and Banach spaces.}
Filomat, 38, 11633--11637 (2024).

\bibitem{EEG2025-2}
Emelyanov, E., Erkurşun-\"{O}zcan, N. and Gorokhova, S. 
{\em d-Operators in Banach lattices.}
Siberian Mathematical Journal, 2025, Vol.66, No.6, pp.1--11.

\bibitem{Halmos1982}
P. R. Halmos,
{\em A Hilbert Space Problem Book.}
Springer (1982)

\bibitem{S2025}
K. Storozhuk,
{\em Disjointly almost trivial discontinuous functionals.}
arxiv.org/2512.07227
\end{thebibliography}
\end{document}